\documentclass{elsarticle}
\usepackage{amsmath}
\usepackage{graphicx}
\usepackage{hyperref}
\usepackage{lineno}
\modulolinenumbers[5]
\usepackage{amsthm}
\usepackage{amssymb}
\usepackage{todonotes}
\usepackage{graphicx}
\usepackage{physics}
\usepackage{geometry}
\usepackage{color}
\usepackage{tikz}
\usepackage{enumerate}

\newtheorem{thm}{Theorem}[section]
\newtheorem{cor}[thm]{Corollary}
\newtheorem{lem}[thm]{Lemma}
\newtheorem{prop}[thm]{Proposition}
\theoremstyle{definition}
\newtheorem{defn}[thm]{Definition}
\newtheorem{example}[thm]{Example}
\newtheorem{rem}[thm]{Remark}
\numberwithin{equation}{section}
\newcommand{\SCH}{Schr\"odinger}

\begin{document}

\begin{frontmatter} 

\author[IAPCM]{
 Xiaoyan 
 ~Su
 \corref{cor1}
}
\ead{suxiaoyan@qq.com}
\author[SICH]{
Shiliang 
~Zhao
}
\ead{zhaoshiliang@scu.edu.cn}
\author[SICH]{
Miao 
~Li 
}
\ead{mli@scu.edu.cn}
\cortext[cor1]{corresponding author}
\address[IAPCM]{Institute of Applied Physics and Computational Mathematics, 100094, P.R.China.}
\address[SICH]{Department of Mathematics, Sichuan University, 610064, P.R.China.}

\title{Local well-posedness of semilinear space-time fractional Schr\"odinger equation}

\begin{abstract}
The semilinear space-time fractional Schr\"odinger equation is considered. First, we give the explicit form for the fundamental solutions by using the Fox $H$-functions in order to to establish some $L^s$ decay estimates. After that, we give some space-time estimates for the
mild solutions from which the local well-posedness is derived on some proper Banach space.
\end{abstract}


\end{frontmatter}

\linenumbers 

\tableofcontents
\section{Introduction}

In this paper, we consider the Cauchy problem for the following semilinear space-time fractional Schr\"odinger equation:
\begin{equation}\label{fse}
 \left \{\begin{split}
  i\partial_t^\alpha u(t,x) &= (-\Delta)^{\frac{\beta}{2} }u(t,x)\pm|u(t,x)|^\theta u(t,x), &(t,x)&\in [0,T)\times \mathbb R^n\\
  u(0,x)&=u_0(x), & x&\in \mathbb R^n.
 \end{split}
 \right. 
\end{equation}Here $\partial_t^\alpha$ is the Caputo fractional derivative with $0<\alpha<1$, $\theta>0$ and $(-\Delta )^\beta$ is the fractional Laplacian with $\beta>0$, which is given by the Fourier multiplier, i.e.
$\mathcal F((-\Delta )^{\frac{\beta}{2}} u)(\xi)=|\xi|^{\beta}\hat{u}(\xi)$. The above equation can be obtained if we fractionalize both the time and space derivatives in the classical Schr\"odinger equation.

In this paper, we prefer to keep the imaginary number $i$ instead of `fractionalizing' it by replacing $i$ with $i^\alpha$. It is a generalization of the equation which has been studied in \cite{Narahari Achar}. And we can derive some interesting decay properties for the fundamental solutions which can be used to establish the local well-posednes, whereas the same approach might not be valid if we replace $i$ with $i^\alpha$.

Though a number of articles have been dedicated to the study of the semilinear Schr\"odinger equation due to its great importance in both mathematics and physics, for example \cite{Bourgain}\cite{Cazenave}\cite{Tao}; the theory about fractional Schr\"odinger equation is quite recent and is far from complete. The space fractional Schr\"odinger equation was introduced  by Laskin in 2000 \cite{Laskin} \cite{Laskin 1}, explaining that this equation can be derived by using the Feynman path integral techniques replacing the Brownian paths  with L\'evy stable paths, and the Markovian character of the solution is still remained. As for the time fractional case, in 2004, Naber first investigated the time fractional Schr\"odinger equation in \cite{Naber}. In that paper, he gave physical reasons to suggest that fractionalizing $i$ results in a better model. On the other hand, the opposite conclusion was reached by Narahari Achar et. al. in the more recent paper \cite{Narahari Achar}. Other works about fractional Schr\"odinger equation include \cite{Dinh},\cite{Bayin 3},\cite{Dong}references therein. Recently, Laskin published a book \cite{Laskin 2} which gives a systematic introduction about both space and time fractional Schr\"odinger equation.

The well-posedness of similar equations like the semilinear \SCH{} equation and semilinear heat equation is well understood (see for instance\cite{Cazenave} and \cite{Miao}\cite{Miao 2}\cite{Miao 3}\cite{Miao 4}). In contrast, there seems to be a difficulty in developing analogous results for the time fractional \SCH{}, mainly due to the loss of time invariance. One result in this direction is given by Grande \cite{Grande}, and he consider the one dimensional space-time fractional \SCH{} equation in Sobolev space with $i^\alpha$ instead of $i$ and gave the local well-posedness of this equation by exploiting the smooth effect. Inspired by the work \cite{Kemppainen} and \cite{Miao}, we establish the local well-posedness in Lebesgue spaces for \eqref{fse} in any dimension. 

In order to get the results in our paper, we mainly rely on properties of $H$-functions. First we establish the $L^s$ estimate for the fundamental solutions using detailed asymptotic analysis of the $H$-functions. After that, we give some space-time estimates for the mild solution. This is then used to construct the Banach space where we can prove the local well-posedness for the space-time fractional Schr\"odinger equation.

 Finally, we want to point out that our method can be applied to the fractional semilinear heat equation which has been studied in \cite{Kemppainen} to establish local well-posedness for this equation.  At the same time, we can use our method to reprove and extend the results in \cite{Hirata}.

This paper is organized as follows. In section 2, we introduce some basic notations and preliminaries. In section 3, we derive the fundamental solutions by using the Mellin transform, Fourier transform and Laplace transform. We then obtain the $L^s$ decay estimates by exploiting the asymptotic behaviors of $H$-functions. In section 4, we establish some space-time estimates for both the homogenous part and non-homogenous part. The local well-posedness results are given in the section 5 by using the contraction mapping principle. Unlike the classical time derivative case, we cannot get the blow up criterion because the Caputo derivative is not invariant under translations in time. Some critical properties of $H$-functions which have been used in our paper are given in appendix.

\section{Preliminaries}
  
  We write $A\lesssim B $ if and only if there is a positive constant $C$ such that $A\leq CB$ which might change from line to line. We write $A(z)\sim B(z)$ as $z\to z_0$ to mean that $A(z) \lesssim B(z)$ and $B(z)\lesssim A(z)$, for $z$ sufficiently close to $z_0$.
  
  For $1\leq p, q \leq \infty$ and $I=[0, T)$, use $L^p(\mathbb R^n) $ to denote the Lebesgue integrable space and $L^p_w (\mathbb R^n)$ to denote the weak $L^p(\mathbb R^n)$ space.  $L^q(0,T;L^p(\mathbb R^n))$ is the space of functions $f$ taking values in $L^p(\mathbb R^n)$ such that 
\[ \| f \|_{L^q(0,T;L^p(\mathbb R^n))} := \left( \int_0^T \|f(x,t)\|_{L^p(\mathbb R^n)} d{t} \right)^{1/q} < \infty \]

 $ C_b (I;L^p(\mathbb R^n))$ is the space of continuous bounded functions taking values in $L^p(\mathbb R^n)$ with finite norm, given by
   \[\|f\|_{C_b(I; L^p(\mathbb R^n))}=\sup_{t\in I}\left(\int_{\mathbb R^n}|f(t,x)|^pdx\right)^{\frac{1}{p}} \]The fractional Riemann-Liouville integral of order $\alpha >0$ is defined as
\[J^{\alpha} f(t)= \frac{1}{\Gamma(\alpha)}\int_0^t(t-s)^{\alpha-1}f(s)\,ds,\]
where \begin{equation*}
g_{\alpha}(t)=\left \{\begin{aligned}
	&\frac{t^{\alpha-1}}{\Gamma(\alpha)}\quad &\text{when} \quad &t>0\quad \\
	&0\quad &\text{when} \quad &t\leq 0 
	\end{aligned}\right.
	\end{equation*}
and $\Gamma(\alpha )$ is the Gamma function.  The fractional Caputo derivative of order $\alpha\in (0,1)$ is defined as
\[(D_t^\alpha f)(t)= J^{1-\alpha}f'(t)=\frac{1}{\Gamma(1-\alpha)}\int_0^t \frac{f'(\tau)}{(t-\tau)^\alpha}ds. \]

  We use the following standard notation for the Fourier transform
   \[ \mathcal F (f(x))(\xi)=\hat{f}(\xi)= \frac{1}{(2\pi)^{\frac{n}{2}}}\int_{\mathbb R^n}e^{-ix\cdot \xi}f(x)dx ,\]
   and the fractional Laplacian is given by the Fourier multiplier, $\mathcal F((-\Delta )^{\frac{\beta}{2}} u(x))=|\xi|^{\beta}\hat{u}(\xi)$.
   
 The Laplace transform with respect to time is given by 
 \[\mathcal L (f(t))(\lambda)=\mathcal	L[f(t), \lambda] =\tilde {f}(\lambda)=\int_0 ^\infty e^{-\lambda t} f(t)dt,\]
 and it is easy to check that $\mathcal L (D_t^\alpha f(t))(\lambda)= \lambda ^\alpha \tilde f(\lambda )-\lambda^{\alpha-1} f(0).$

The Mellin transform of a function is defined by 
\[\mathcal M (f(r))(s)=\mathcal M[f(t), s]=f^*(s)=\int_0^{\infty} f(r)r^{s-1}dr,\quad c_1\leq \Re s \leq c_2\] 
as long as the above integral is valid.
In general, the Mellin transform of $f$ can be further defined on an open subset of $\mathbb C$ containing $\{ \Re s \in (c_1,c_2) \}$ by analytic continuation. We refer the reader to section 4.3 of \cite{Bleisten} for details.

The Mellin transform has the following properties which will be used in our paper.
\begin{align*}
	f ( a r ) &\stackrel {\mathcal M } { \leftrightarrow } a ^ { - s } f ^ { * } ( s ) , a > 0\\
r ^ { a } f ( r ) &\stackrel {\mathcal M }\leftrightarrow f ^ { * } ( s + a )\\
f ( r ^ { p } ) &\stackrel {\mathcal M } \leftrightarrow \frac { 1 } { | p | } f ^ { * } ( s / p ) , p \neq 0
\end{align*}
 Here we want to point out the first property might not be valid for all $a\in \mathbb C$. But $\mathcal M [f(e^{i\theta }r)]=e^{-i\theta s}f^*(s)$ can be true for some  functions by using the Cauchy integral theorem.
In fact, we can prove that when $f(t)=\frac{1}{t^\alpha+1} (\alpha>0)$, this is valid, which is  just the result we will use in our proof.

Now we give one result which we will use later in the following example.

\begin{example}\label{example}
Consider the function $f(t)=\frac{1}{t^\alpha+1} (\alpha>0)$, the Mellin transform is given by
\begin{align}\label{mf}
f^*(s)=\int_0^\infty \frac{t^{s-1}}{t^\alpha+1}dt=\frac{\Gamma(\frac{s}{\alpha})\Gamma(1-\frac{s}{\alpha}) }{\alpha}
\end{align}
	and it is absolutely convergent for $0<\Re s<\alpha$. Then we can extend the Mellin transform of $f(t)$ to any $s\neq n\alpha$ by the right hand side of \eqref{mf}, where $n$ is an integer.
	
	Now we claim that the Mellin transform of $f(e^{i\theta} t)$ with $-\pi <\theta<\pi $ is 
	\begin{align}
	[f(e^{i\theta}t)]^*(s)= \frac{e^{-i\theta s}\Gamma(\frac{s}{\alpha})\Gamma(1-\frac{s}{\alpha}) }{\alpha}.
	\end{align}

	Now we cut the complex plane from $(-\infty, 0]$ and take the principal branch of $z^\alpha$. Let us denote 
		  $\Gamma_1=\{re^{i\theta}: 0\leq r \leq \infty\}$,
		$\Gamma_R=\{Re^{i\sigma}: 0\leq \sigma \leq \theta \}$,
		  $\Gamma_0^R=[0, R]$,
		 $\Gamma_1^R=\{re^{i\theta}: 0\leq r \leq R\}$, as in the following graph.

	
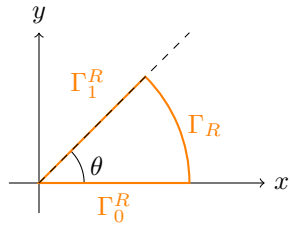
\begin{center}\begin{tikzpicture}[scale=2]

    \draw[->] (-0.2,0) -- (1.5,0)  node[right]  {$x$};
    \draw[->] (0,-0.2) -- (0,1) node[above] {$y$};
    
    \draw (0.3,0) 
    arc (0:22.5:0.3) node[right] {$\theta$}
    arc (22.5:45:0.3);
    \draw[thick,orange]  (0,0) 
    -- (0.5,0) node[below] {$\Gamma_0^R$}
    -- (1,0) 
    arc (0:22.5:1) node[right] {$\Gamma_R$}
    arc (22.5:45:1)
    -- (0.5,0.5) node[above left] {$\Gamma_1^R$}
    -- cycle ;
    
    \draw[dashed] (0,0) -- (1,1);
    \end{tikzpicture}
\end{center}

Then we have 
	\begin{align*}
		[f(e^{i\theta}t)]^*(s)&=\int_0^\infty \frac{t^{s-1}}{(e^{i\theta}t)^\alpha+1}dt\\
		&=e^{-i\theta s}\int_{\Gamma_1} \frac{\mu^{s-1}}{\mu^\alpha +1}d\mu
	\end{align*}

Since $ f(z) = O(z^{s-\alpha-1})$ as $z$ goes to infity, it  is easily shown that the integral along $\Gamma_R$ vanishes as $R$ tends to $\infty$, which implies the result by Cauchy's theorem
as long as $0<\Re s<\alpha$. Finally we extend the $f^*(e^{i\theta}t)(s)$ to any $s\neq n\alpha$. 
		
\end{example}

Notice that by the definitions, we have 
\begin{align}\label{mlr}
\mathcal M[f(t), s]=\frac{1}{\Gamma(1-s)} \mathcal M[\mathcal L[f(t), \lambda ], 1-s]
\end{align}
as long as both sides make sense.

The Fox $H$-functions \cite{Kilbas} have a crucial role in our theory because of their natural connection to the fractional calculus. Here we give their definition  and their Mellin transform. 
  
  \begin{defn} For integers $m, n, p, q$ such that $0\leq m\leq q, 0\leq n \leq p$, for $a_i, b_j\in \mathbb C$ and for $\alpha_i, \beta_j\in \mathbb {R}_{+}=(0, \infty )(i=1,2,\dots,p; j=1,2,\dots, q )$, the $H$-function $H_{p,q}^{m,n}(z)$ is defined via a Mellin-Barnes type integral in the form
  \[H_{p,q}^{m,n}(z):=H_{p,q}^{m,n}\left[z \middle |\begin{smallmatrix}(a_1,\alpha_1 ), \dots, (a_p, \alpha_p)\\
  (b_1,\beta_1),\dots, (b_q, \beta_q)	
  \end{smallmatrix}\right ]=\frac{1}{2\pi i}\int_{\mathcal L}\mathcal{H}_{p,q}^{m,n}(s)z^{-s}ds,
 \]
 with
 \[\mathcal{H}_{p,q}^{m,n}(s)=\frac{\prod_{j=1}^m\Gamma(b_j+\beta_js)\prod_{i=1}^{n}\Gamma(1-a_i-\alpha_i s )}{\prod_{j=n+1}^p\Gamma(a_i+\alpha_is)\prod_{j=m+1}^{q}\Gamma(1-b_j-\beta_j s )}, \]
  	and $\mathcal L$ is the infinite contour in the complex plane which separate the poles
  	\[b_{jl}=\frac{-b_j-l}{\beta_j} \quad (j=1,\dots,m; l=0,1,2,\dots ), \]
  	of the Gamma functions $\Gamma(b_j+\beta_j s)$ to the left of $\mathcal L$ and the poles 
  	\[a_{ik}=\frac{1-a_i+k}{\alpha_i}\quad (i=1,\dots, n; k=0, 1,2,\dots), \]
  	of the Gamma functions $\Gamma(1-a_i-\alpha_i s)$ to the right of $\mathcal L$.

  \end{defn}
  The properties of the $H$-function $H_{p,q}^{m,n}(z)$ depend on the parameters of $a^{*}$,  $\Lambda$, $\mu$ and $\delta$ which are given by 
  \begin{align*}
  	a^{*}&=\sum_{i=1}^{n}\alpha_i-\sum_{i=n+1}^{p}\alpha_i +\sum_{j=1}^{m}\beta_j-\sum_{j=m+1}^{q}\beta_j,\\
  	\Lambda &=\sum_{j=1}^{q}\beta_j-\sum_{i=1}^{p}\alpha_ i,\\
  	\mu &= \sum_{j=1}^q b_j-\sum_{i=1}^{p} a_i+ \frac{p-q}{2} .\\
  	\delta &=\prod_{i=1}^p \alpha_ i^{-\alpha_ i}\prod_{j=1}^q \beta_{j}^{\beta_j}
  \end{align*}
  
  In the following lemma, we give the Mellin transform for the $H$-function in the case $a^*>0$, see Theorem 2.2 \cite{Kilbas}.
  
  \begin{lem} Let $a^{*}> 0$ and $s\in \mathbb C$ such that 
\[-\min_{1\leq j\leq m} \left[\frac{\Re(b_j)}{\beta_j} \right]< \Re (s)<\min_{1\leq i\leq n}\left [\frac{1-\Re(a_i)} {\alpha_i} \right],\]
 then the Mellin transform of the $H$-function exists and 
  \begin{align} 
      \mathcal M[H_{p, q}^{m,n}(z), s]= \mathcal{H}_{p,q}^{m,n}(s)=\frac{\prod_{j=1}^m\Gamma(b_j+\beta_js)\prod_{i=1}^{n}\Gamma(1-a_i-\alpha_i s )}{\prod_{j=n+1}^p\Gamma(a_i+\alpha_is)\prod_{j=m+1}^{q}\Gamma(1-b_j-\beta_j s )}. \label{lem-mellin-H}
  \end{align}
  \end{lem} 
  \

\section{The fundamental solutions and its $L^s$ estimates}

In order to give the explicit form of the fundamental solutions for  \eqref{fse}, we consider the following  two equations:
\begin{equation}\label{fse1}
	\left\{\begin{split}
		i\partial_t^\alpha u(t,x) &= (-\Delta)^{\frac{\beta}{2} }u(t,x),  &(t,x)&\in [0,T)\times \mathbb R^n\\
		u(0,x)&=\delta (x), & x&\in \mathbb R^n,
	\end{split}
	\right.	
\end{equation}
and
\begin{equation}\label{fse2}
	\left \{\begin{split}
		i\partial_t^\alpha u(t,x) &= (-\Delta)^{\frac{\beta}{2} }u(t,x)+\delta(t)\delta(x), &(t,x)&\in [0,T)\times \mathbb R^n\\
		u(t,x)&=0, & x&\in \mathbb R^n.
	\end{split}
	\right.	
\end{equation}
We use $S_{\alpha,\beta}(t,x)$ and $P_{\alpha,\beta}(t,x)$ to denote the solution for \eqref{fse1} and \eqref{fse2} respectively, then the mild solution of \eqref{fse} is given by 
\begin{align*}
	u(x,t)&=S_{\alpha, \beta}(t,x)\ast u_0(x)\pm\int_0^t \int_{\mathbb R^n }P_{\alpha,\beta}(t-\tau , x-y)|u(\tau, y)|^\theta u(\tau, y)dy d\tau,
\end{align*}
by the standard Duhamel principle, we now compute the  explicit form of $S_{\alpha,\beta}(t,x)$ and $P_{\alpha,\beta}(t,x)$ as a first step.

\subsection{The explicit form of the fundamental solutions}

 In the first part of this section,  we derive the explicit form of $S_{\alpha,\beta}(t,x)$ and $P_{\alpha,\beta}(t,x)$ by using  
 $H$-functions.

We start by computing the solution to \eqref{fse1}. Taking the Laplace transform in time and the Fourier transform in  space,
\begin{align}\label{flS}
	\tilde {\hat{S}}_{\alpha,\beta}(\lambda , \xi)=\frac{i\lambda^{\alpha-1}}{i\lambda^\alpha-|\xi|^\beta}= \frac{\lambda^{\alpha-1}}{\lambda^\alpha+i|\xi|^\beta}.
	\end{align}
Using the relationship \eqref{mlr},
we can obtain the Mellin transform in time and Fourier transform in space  as follows:
\begin{align*}
	\hat{S}^*_{\alpha,\beta}(s,\xi)&=\frac{1}{\Gamma(1-s)} \int_0^\infty \frac{\lambda^{\alpha-1-s}}{\lambda^{\alpha}+i|\xi|^{\beta}}d\lambda,
\end{align*}
combining the result in Example \eqref{example} and the properties of Mellin transform, we have 
\begin{align*}
\hat{S}^*_{\alpha,\beta}(s,\xi)	&=\frac{i^{-\frac{s}{\alpha}}|\xi|^{-\frac{\beta}{\alpha} s}}{\alpha } 
 	\frac{\Gamma(\frac{\alpha-s}{\alpha} )\Gamma(1-\frac{\alpha-s}{\alpha} )}{\Gamma(1-s)}.	
\end{align*}
Then taking the inverse Fourier transform in space and  employing the formula $\widehat{ |x|^z }= \frac{\Gamma(\frac{z+n}{2})}{\Gamma(\frac{-z}{2})}|\xi|^{-z-n}$, we get 
\[{S}^*_{\alpha,\beta}(s,x)=\frac{i^{-\frac{s}{\alpha} }|x|^{\frac{\beta}{\alpha} s-n}\Gamma(\frac{\alpha-s}{\alpha} )\Gamma(1-\frac{\alpha-s}{\alpha} )\Gamma(\frac{n-\frac{\beta}{\alpha}s}{2} )}{\alpha\Gamma(1-s)\Gamma(\frac{\frac{\beta}{\alpha}s}{2} )}\]
Taking the inverse Mellin transform and using the properties of the $H$-functions yields that
\[{S}_{\alpha,\beta}(t, x)=\frac{1}{ |x|^n} H_{2,3}^{2,1}\left( -i|x|^{\beta}t^{-\alpha}\middle|\begin{smallmatrix}
	&(1,1), &(1,\alpha )\\
	 &(\frac{n}{2}, \frac{\beta}{2}), &(1,1), &(1,\frac{\beta}{2} )
\end{smallmatrix} \right)\]

With the same idea in mind, we can obtain the form of $P_{\alpha, \beta}(t,x)$, which is given by 
\[{P}_{\alpha,\beta}(t, x)=\frac{it^{\alpha-1}}{ |x|^n} H_{2,3}^{2,1}\left(-i|x|^{\beta}t^{-\alpha}\middle |\begin{smallmatrix}
	&(1,1), &(\alpha ,\alpha )\\
	 &(\frac{n}{2}, \frac{\beta}{2}),&(1,1 ), &(1,\frac{\beta}{2} )
\end{smallmatrix} \right ).\]

\begin{rem} In fact, we can divide the fundamental solutions into real and imaginary parts, i.e. we can rewrite ${S}_{\alpha,\beta}(t, x)={S}_{\alpha,\beta}^1(t, x)-i{S}_{\alpha,\beta}^2(t, x)$ where ${S}_{\alpha,\beta}^1(t, x), {S}_{\alpha,\beta}^2(t, x)$ are both real-valued functions. This can be obtained if we rewrite \eqref{flS} as
\[\tilde {\hat{S}}_{\alpha,\beta}(\lambda , \xi)=\frac{i\lambda^{\alpha-1}}{i\lambda^\alpha-|\xi|^\beta}. =\frac{\lambda^{2\alpha-1} }{\lambda^{2\alpha}+|\xi|^{2\beta}}-i \frac{\lambda^{\alpha-1}|\xi|^{\beta} }{\lambda^{2\alpha}+|\xi|^{2\beta}}\nonumber,
\]	
and proceed with the same method. Finally, using the $H$-functions, we have  
\[{S}_{\alpha,\beta}^1(t, x)=\frac{1}{2 |x|^n} H_{2,3}^{2,1}\left(|x|^{\beta}t^{-\alpha}\middle|\begin{smallmatrix}
	&(1,\frac{1}{2}), &(1,\alpha )\\
	& (1,\frac{1}{2}), &(\frac{n}{2}, \frac{\beta}{2}),&(1,\frac{\beta}{2} )
\end{smallmatrix} \right),\]
\[{S}_{\alpha,\beta}^2(t, x)=\frac{1}{2 |x|^n} H_{2,3}^{2,1}\left(|x|^{\beta}t^{-\alpha}\middle|\begin{smallmatrix}
	&(\frac{1}{2},\frac{1}{2}), &(1,\alpha )\\
	& (\frac{1}{2},\frac{1}{2}), &(\frac{n}{2}, \frac{\beta}{2}),&(1,\frac{\beta}{2} )
\end{smallmatrix}\right)\]

Similarly, we can rewrite ${P}_{\alpha,\beta}(t, x)=-{P}_{\alpha,\beta}^1(t, x)-i{P}_{\alpha,\beta}^2(t, x)$, with \[{P}_{\alpha,\beta}^1(t, x)=\frac{t^{\alpha-1}}{2 |x|^n} H_{2,3}^{2,1}\left(|x|^{\beta}t^{-\alpha}\middle |\begin{smallmatrix}
	&(\frac{1}{2},\frac{1}{2}), &(\alpha ,\alpha )\\
	& (\frac{1}{2},\frac{1}{2}), &(\frac{n}{2}, \frac{\beta}{2}),&(1,\frac{\beta}{2} )
\end{smallmatrix} \right),\]
\[{P}_{\alpha,\beta}^2(t, x)=\frac{t^{\alpha-1}}{2 |x|^n} H_{2,3}^{2,1}\left(|x|^{\beta}t^{-\alpha} \middle |\begin{smallmatrix}
	&(1,\frac{1}{2}), &(\alpha ,\alpha )\\
	& (1,\frac{1}{2}), &(\frac{n}{2}, \frac{\beta}{2}),&(1,\frac{\beta}{2} )
\end{smallmatrix} \right ).\]

\end{rem}

For convenience, make the definition $z:=i|x|^\beta t^{-\alpha} $.  Then, we have  $|\arg z |= |\arg (i|x|^\beta t^{-\alpha})|=\frac{\pi}{2}$. Also, we remind the reader that the parameters $\Lambda,a^*$ from the appendix for these $H$-functions are $ \Lambda =\beta-\alpha, a^*=2-\alpha$.

\subsection{The asymptotic behavior of the fundamental solutions}
In this subsection, in order to give some basic properties of $S_{\alpha, \beta}$ and $P_{\alpha, \beta }$, we will rely on the asymptotic properties of the Fox $H$-function, as proved in the Appendix.

In the following lemma, we give the asymptotics of $S_{\alpha, \beta}$ and $P_{\alpha, \beta }$ as $|z|\to 0$, and $|z|\to \infty$.
\begin{lem}\label{asy}  Let $0<\alpha <1$ and $\beta>0$. Then the function $S_{\alpha, \beta}, P_{\alpha, \beta}$ has the following asymptotic behavior :
\begin{enumerate}
		\item \begin{enumerate}
	\item If $|z|\geq 1$, then we have 
	\[S_{\alpha, \beta}(t,x)\begin{cases}
		\lesssim |x|^{-n}\exp{-c|x|^{\frac{\beta}{\beta-\alpha}}t^{-\frac{\alpha}{\beta}}}, &  \beta= 2k,\\
		\sim t^{\alpha}|x|^{-n-\beta} & \beta \neq 2k,
	\end{cases} \]
	In particular, we have $S_{\alpha, \beta}(t,x)\lesssim t^{\alpha}|x|^{-n-\beta}$ for any $\beta>0$.
	\item If $|z|\leq 1$, then 
	\[S_{\alpha, \beta}(t,x)\sim \begin{cases}
		t^{-\frac{\alpha n}{\beta} }, &  \beta>n,\\
		t^{-\alpha}(|\log(t^{-\alpha}|x|^\beta)|+1),& \beta=n,
\\ 
		t^{-\alpha}|x|^{-n+\beta}, & \beta<n.
	\end{cases}  \]
		      \end{enumerate}
	\item \begin{enumerate}
	\item If $|z|\geq 1$, then we have 
	\[P_{\alpha, \beta}(t,x)\begin{cases}
		\lesssim |x|^{-n}t^{\alpha-1}\exp{-c|x|^{\frac{\beta}{\beta-\alpha}}t^{-\frac{\alpha}{\beta}}}, &  \beta= k,\\
		\sim t^{2\alpha-1}|x|^{-n-\beta} & \beta \neq k,
	\end{cases} \]
	In parrticular, we have $S_{\alpha, \beta}(t,x)\lesssim t^{2\alpha-1}|x|^{-n-\beta}$.
	\item If $|z|\leq 1$, then 
	\[P_{\alpha, \beta}(t,x)\sim \begin{cases}
		t^{\alpha-1-\frac{\alpha n}{\beta} }, &  \beta>\frac{n}{2},\\
		t^{-\alpha-1}(|\log(t^{-\alpha}|x|^\beta)|+1),& \beta=\frac{n}{2},
\\ 
		t^{-\alpha-1}|x|^{-n+2\beta}, & \beta<\frac{n}{2}.
	\end{cases}  \]
		      \end{enumerate}
	
	\end{enumerate}	  
\end{lem}

\begin{proof}  The asymptotics will follow from a careful computation of residues. We first prove the results for the $S_{\alpha,\beta}$.
For the case $|z|\geq 1$, from \eqref{lem-mellin-H}, the Mellin transform of $H_{2,3}^{2,1}$ is 
\[\mathcal H_{2,3}^{2,1}(s)=\mathcal M\left(H_{2,3}^{2,1}\left(z\middle|\begin{smallmatrix}
	&(1,1), &(1 ,\alpha )\\
	& (1,1), &(\frac{n}{2}, \frac{\beta}{2}),&(1,\frac{\beta}{2} )
\end{smallmatrix} \right )\right )(s)=\frac{\Gamma(\frac{n}{2}+ \frac{\beta}{2}s)\Gamma(1+s)\Gamma(-s)}{\Gamma(1+\alpha s)\Gamma(-\frac{\beta}{2}s)}. \]
 Since we have $ a^*= 2-\alpha>1 >0, |\arg z|=\frac{\pi}{2}\leq a^*\frac{\pi}{2}$,  then by Lemma \ref{A1}, we have $H_{2,3}^{2,1}(z)\sim \sum_{k=0}^{\infty}h_k z^{-k}$, where 
\begin{align*}
 h_k&=\lim_{s\to a_{1k}}[-(s-a_{1k})\mathcal H_{2,3}^{2,1}(s)]\nonumber \\
 &= \frac{(-1)^k}{k!\alpha_1}\frac{\Gamma(b_1+(1-a_1+k)\frac{\beta_1}{\alpha_1} )\Gamma(b_2+(1-a_1+k)\frac{\beta_2}{\alpha_1} )}{\Gamma(a_2+(1-a_1+k)\frac{\alpha_2}{\alpha_1} )\Gamma(1-b_3-(1-a_1+k)\frac{\beta_3}{\alpha_1} )}.
\end{align*}
Now we calculate $h_k$. 

When $\beta \neq 2k$, we obtain that $h_0=\frac{\Gamma(\frac{n}{2}) \Gamma(1)}{\Gamma(1)\Gamma(0)}=0, $ and $h_1=\frac{\Gamma(\frac{n}{2}+\frac{\beta}{2}) \Gamma(2)}{\Gamma(1+\alpha)\Gamma(-\frac{\beta}{2})}\neq0, $ therefore
the leading term is $h_1 z^{-1}$, i.e., 
\[H_{2,3}^{2,1}(z)\sim z^{-1},\quad  z\to \infty,\]
proving the second result of (a).

When $\beta= 2k$, all the $h_k$ vanish, so we cannot use Lemma \ref{A1} anymore. In this case, we have the expression $\mathcal H_{2,3}^{2,1}(s)=\frac{\Gamma(\frac{n}{2}+ \frac{\beta}{2}s)\Gamma(1+s)\prod_{j=1}^{\frac{\beta}{2}-1}\Gamma(\frac{2j}{\beta} +s)}{\Gamma(1+\alpha s)}.$ 
Therefore, by Lemma \ref{A4}, we have the following asymptotic  as $|z|\to \infty$,
\[H_{1,\frac{\beta}{2}+1}^{\frac{\beta}{2}+1,0}(z)=O\left(|z|^{[\frac{n}{2} +\frac{1}{4} ]/(\beta-\alpha)}\exp \left \{(\beta-\alpha)\left(\frac{|z|}{{\alpha}^{-\alpha}\beta^\beta}\right )^{1/(\beta-\alpha )}\cos (\pi+\frac{\arg z}{\beta-\alpha}) \right \}\right ).\]
In our case $\arg z= \frac{\pi}{2},$ and owing to $\beta=2k-\alpha>1$,  we obtain that 
$H_{1,\frac{\beta}{2}+1}^{\frac{\beta}{2}+1,0}(z)=O(\exp{-c|z|^{1/(\beta-\alpha)}} )$, thus the first result in (a) has been verified.

(b) Next we deal with the case  $|z|\leq 1$. In order to obtain the result, we need to use the asymptotic for $H_{2,3}^{2,1}(z)$ when $|z|\to 0$. Notice that $\mathcal H_{2,3}^{2,1}(s)$ has two poles at $s=-1$
   and $s=-\frac{n}{\beta}$. By employing   Lemma \ref{A2} , the asymptotic behavior is therefore determined by the largest pole, which depends on the relative size of $\beta$ and $n$:
   \begin{enumerate}[(i)]
   	\item When $\beta>n$, we have $-\frac{n}{\beta}>-1$, so $H_{2,3}^{2,1}(z)\sim z^{\frac{n}{\beta}}, |z|\to 0$;
   	\item When $\beta<n$, we have $-\frac{n}{\beta}<-1$, so $H_{2,3}^{2,1}(z)\sim z, |z|\to 0$;
   	\item When $\beta=n$, $\mathcal H_{2,3}^{2,1}(s)$ has a second order pole at $s=-1$. By Lemma \ref{A3}, we conclude that $H_{2,3}^{2,1}(z)\sim z \ln z, |z|\to 0$.
   \end{enumerate}
   Recalling that $z=i|x|^{\beta}t^{-\alpha}$, we conclude the proof of (b).

 The proof for  $P_{\alpha, \beta}(t,x)$ is similar. Notice that the Mellin transform for $H_{2,3}^{2,1}\left (z\middle  |\begin{smallmatrix}
	&(1,1), &(\alpha ,\alpha )\\
	&(\frac{n}{2}, \frac{\beta}{2}), & (1,1), &(1,\frac{\beta}{2} )
\end{smallmatrix} \right)$ is just $\frac{\Gamma(\frac{n}{2}+ \frac{\beta}{2}s)\Gamma(1+s)\Gamma(-s)}{\Gamma(\alpha+\alpha s)\Gamma(-\frac{\beta}{2}s)}$.
When $|z|\geq 1$, the proof is the almost exactly the same as the analogous case in (1) so we omit the details.

When $|z|\leq 1$, $\frac{\Gamma(\frac{n}{2}+s)\Gamma(1+s)\Gamma(-s)}{\Gamma(\alpha+\alpha s)\Gamma(-\frac{\beta}{2}s)}$ has two poles at $s=-\frac{n}{\beta}$ and $s=-2$ (notice that $s=-1$ is not a pole, since $\Gamma(\alpha-\alpha )=\Gamma(1-1)$), so again by Lemma \ref{A2}, we have three cases:
\begin{enumerate}[(i)] 
   	\item When $\beta>\frac{n}{2}$, we have $-\frac{n}{\beta}>-2$, so $H_{2,3}^{2,1}(z)\sim z^{\frac{n}{\beta}}, |z|\to 0$;
   	\item When $\beta<n$, we have $-\frac{n}{\beta}<-2$, so $H_{2,3}^{2,1}(z)\sim z^2, |z|\to 0$;
   	\item When $\beta=\frac{n}{2}$, $\mathcal H_{2,3}^{2,1}(s)$ has a second order pole at $s=-2$,  therefore we conclude that $H_{2,3}^{2,1}(z)\sim z^2 \ln z, |z|\to 0$.
   \end{enumerate}
   Recalling that $z=i|x|^{\beta}t^{-\alpha}$, we have finished the proof of the lemma.

\end{proof}

\subsection{The $L^s$ estimate of fundamental solutions}

After deriving the asymptotic behavior in the previous part, now we are ready to give the $L^s$ estimate for the fundamental solutions of $S_{\alpha,\beta}, P_{\alpha,\beta}$.

\begin{lem} \label{lem-Ls-estimate} Let $t>0$ be given. Then:
\begin{enumerate}
	\item \begin{enumerate}
	\item If $\beta>n$, then $S_{\alpha,\beta}(t,\cdot)\in L^s(\mathbb R^n ) $ for all $s\in [1, \infty]$ and we have \begin{align}\label{S estimate}
		\|S_{\alpha,\beta}(t,\cdot)\|_{ L^s(\mathbb R^n )}\lesssim t^{-\frac{\alpha n}{\beta}(1-\frac{1}{s})}, 
		\end{align}
	\item If $\beta=n$, then $S_{\alpha,\beta}(t,\cdot)\in L^s(\mathbb R^n ) $ and the inequality \eqref{S estimate} holds for all $s\in [1, \infty)$.
	\item If $\beta <n$, then $S_{\alpha,\beta}(t,\cdot)\in L^s(\mathbb R^n ) $ and the inequality \eqref{S estimate} holds for all $s\in [1, \frac{n}{n-\beta})$. In addition, we have that  $S_{\alpha,\beta}(t,\cdot) \in {L_{w}^{\frac{n}{n-\beta}}({\mathbb R^n})}$, and 
	
	\begin{align*}
		 \|S_{\alpha,\beta}(t,\cdot) \|_{L_{w}^{\frac{n}{n-\beta}}({\mathbb R^n})}\lesssim  t^{-\alpha}.
	\end{align*}
\end{enumerate}

	\item 
	\begin{enumerate}
	\item If $\beta>\frac{n}{2}$, then $P_{\alpha,\beta}(t,\cdot)\in L^s(\mathbb R^n ) $ for all $s\in [1, \infty]$ and we have \begin{align}\label{P estimate}
		\|P_{\alpha,\beta}(t,\cdot)\|_{ L^s(\mathbb R^n )}\lesssim t^{\alpha-1-\frac{\alpha n}{\beta}(1-\frac{1}{s})},
	\end{align}
	\item If $\beta=\frac{n}{2}$, then $P_{\alpha,\beta}(t,\cdot)\in L^s(\mathbb R^n ) $ and the inequality \eqref{P estimate} holds for all $s\in [1, \infty)$.
	\item If $\beta <\frac{n}{2}$, then $P_{\alpha,\beta}(t,\cdot)\in L^s(\mathbb R^n ) $ and the inequality \eqref{P estimate} holds for all $s\in [1, \frac{n}{n-2\beta})$. In addition, we have that $P_{\alpha,\beta}(t,\cdot) \in {L_{w}^{\frac{n}{n-2\beta}}({\mathbb R^n})}$, and 
	
	\begin{align*}
		 \|P_{\alpha,\beta}(t,\cdot) \|_{L_{w}^{\frac{n}{n-2\beta}}({\mathbb R^n})}\leq t^{-1}.
	\end{align*}	
	\end{enumerate}
	\end{enumerate}
\end{lem}
\begin{proof}(1)  For $s=\infty$,  we have by Lemma \ref{asy} that $S_{\alpha,\beta}(t,\cdot) \in L^{\infty}(\mathbb R^n)$ for all $t>0$, provided $\beta>n$, and moreover we have the estimate 
 \[\|S_{\alpha,\beta}(t,\cdot)\|_{L^{\infty}(\mathbb R^n)}\lesssim t^{-\frac{\alpha n}{\beta}}. \]

If $1\leq s<\infty$, we decompose the $L^s$-integral of $S_{\alpha,\beta}(t,x)$ as
\[\|S_{\alpha,\beta}(t,\cdot)\|_{L^s}^s\leq \int_{|x|\geq t^{\frac{\alpha}{\beta}}}|S_{\alpha,\beta}(t,x)|^s dx +\int_{|x|\leq t^{\frac{\alpha}{\beta}}}|S_{\alpha,\beta}(t,x)|^s dx.\]	
In view of Lemma 3.1, for the first term, we have for all $1\leq s \leq \infty$ that
\begin{align*}
	\int_{|x|\geq t^{\frac{\alpha}{\beta}}}|S_{\alpha,\beta}(t,x)|^s dx & \lesssim \int_{|x|\geq t^{\frac{\alpha}{\beta}}}t^{\alpha s}|x|^{-ns-\beta s}dx \nonumber \\
	&\lesssim \int_{t^{\frac{\alpha}{\beta}}}^{\infty} t^{\alpha s}r^{-ns-\beta s} r^{n-1}ds\lesssim t^{-\frac{\alpha n}{\beta}(s-1)}.
\end{align*}
Therefore, we have 
\begin{equation*}
	\left (\int_{|x|\geq t^{\frac{\alpha}{\beta}}}|S_{\alpha,\beta}(t,x)|^s dx\right  )^{\frac{1}{s}}\lesssim t^{-\frac{\alpha n}{\beta}(1-\frac{1}{s} )}.\end{equation*}
The estimate on the second term is bounded differently based on the relationship of $\beta$ and $n$.
\begin{enumerate}[(a)]
	\item If $\beta>n$, then we have 
	\begin{align*}
	\int_{|x|\leq t^{\frac{\alpha}{\beta}}}|S_{\alpha,\beta}(t,x)|^s dx\lesssim 
		\int_0^{t^{\frac{\alpha}{\beta}}}t^{-\frac{\alpha ns}{\beta}}r^{n-1}dr\lesssim  t^{-\frac{\alpha n}{\beta}(s-1)}.
	\end{align*}
	\item If $\beta=n$, then we have 
	\begin{align}
		\int_{|x|\leq t^{\frac{\alpha}{\beta}}}|S_{\alpha,\beta}(t,x)|^s dx &\lesssim \int_{|x|\leq t^{\frac{\alpha}{\beta}}} t^{\alpha s}(\ln(|x|^\beta t^{-\alpha})+1)^s dx\nonumber \\
		&\lesssim \int_0^{t^{\frac{\alpha}{\beta}}}t^{-\frac{\alpha ns}{\beta}}t^{\alpha s}(\ln(r^\beta t^{-\alpha})+1)^s r^{n-1}dr \nonumber\\
		&\lesssim t^{-\alpha s+\alpha}\int_0^1(\ln(\rho ^\beta t^{-\alpha})+1)^s \rho ^{n-1}d\rho \lesssim t^{-\frac{\alpha n}{\beta}(s-1)}\nonumber.
	\end{align}
	for all $1\leq s<\infty$. In the last line, we use the fact that $\beta=n$.
	\item If $\beta <n$, then we have 
	\begin{align}
	\int_{|x|\leq t^{\frac{\alpha}{\beta}}}|S_{\alpha,\beta}(t,x)|^s dx
	&\lesssim \int_{|x|\leq t^{\frac{\alpha}{\beta}}}t^{-\alpha s}|x|^{-ns+\beta s}dx \nonumber\\
	&\lesssim \int_0^{t^{\frac{\alpha}{\beta} }}t^{-\alpha s}r^{-ns+\beta s}r^{n-1}dr\lesssim  t^{-\frac{\alpha n}{\beta}(s-1)}\nonumber.
	\end{align}
	Note that the above integral converges whenever $1\leq s<\frac{n}{n-\beta}$.
	\end{enumerate}
This completes the proof for the $L^s$ estimates. Finally, for the weak $L^s$ estimate, we set $s=\frac{n}{n-\beta}$ and estimate the distribution function,
\begin{align*}
	|\{x\in \mathbb R^d: |S_{\alpha,\beta}(t,x)|>\lambda \}|&\leq |\{x\in \mathbb R^d: |S_{\alpha,\beta}(t,x)|>\lambda\}\cap \{x\in \mathbb R^d:\chi\{|z|\leq 1\} \}|\\
	 &+\leq|\{x\in \mathbb R^d: |S_{\alpha,\beta}(t,x)|>\lambda\}\cap \{x\in \mathbb R^d: \chi\{|z|\geq  1\}|\\
	 &:= A_1+A_2.
\end{align*}
Now we estimate $A_1$ and $A_2$ separately. For $A_1$, we have
 \begin{align*}
 	A_1 &\leq |\{x\in \mathbb R^d: \lambda \leq C t^{-\alpha}|x|^{-d+\beta} \}|
 	\\
 	&= |\{x\in \mathbb R^d: |x| \leq C( t^{-\alpha}\lambda^{-1} )^{\frac{1}{d-\beta}} \}| \leq C (t^{-\alpha}\lambda^{-1} )^{\frac{d}{d-\beta}}
 \end{align*}
 hence we have $\|S_{\alpha,\beta}(t,x)\chi(|z|\leq 1)\|_{L_{w}^{\frac{n}{n-\beta}}({\mathbb R^n})} \lesssim t^{-\alpha}$.

 As for $A_2$, we have
\[\|S_{\alpha,\beta}(t,x)\chi(|z|\geq 1)\|_{L_{w}^{\frac{n}{n-\beta}}({\mathbb R^n})}\leq \|S_{\alpha,\beta}(t,x)\chi(|z|\geq 1)\|_{L^{\frac{n}{n-\beta}}({\mathbb R^n})}\leq Ct^{-\alpha}.\]

This concludes the proof of (1).

 For (2), the proofs of the estimates for $P_{\alpha,\beta}$ are similar to those for $S_{\alpha,\beta}$, so we omit the details.
\end{proof}

\section{The space-time estimate for the mild solution}
In this section, we discuss the decay rate for the mild solution of the equation \eqref{fse}. Before giving our main result, we give the definition of the mild solution.

\begin{defn}
	Let $u_0$ and $f$ be the Lebesgue measurable functions on $\mathbb R^n$ and $[0, \infty)\times \mathbb R^n$, respectively. The function $u$ defined by
	\begin{align}\label{mild solution}
		u(x,t)&=\int_{\mathbb R^n}S_{\alpha,\beta}(t, x-y)u_0(y)dy\pm\int_0^t\int_{\mathbb R^n}P_{\alpha,\beta}(t-\tau, x-y)f(\tau, y)dy d\tau\\\nonumber
		&=:S(t)u_0\pm Gf,
	\end{align}
	is called the mild solution of the Cauchy problem \eqref{fse} if the integral is well-defined. Here, $f(u)=|u|^\theta u$.
\end{defn}

As the application of the above lemma, we have the following decay result for mild solutions.

\begin{prop}\label{Lp estimate for h}
	Assume that $u_0 \in L^r (\mathbb R^n)$ with $1\leq r \leq \infty$. Then $S(t)u_0\in L^p(\mathbb R^n)$, and the estimate
    \begin{equation}\label{basic1}
			\|S(t)u_0\|_{L^{p }(\mathbb R^n)}\lesssim t^{-\frac{\alpha n}{\beta}(\frac{1}{r}-\frac{1}{p} )} \|u_0\|_{L^{r}(\mathbb R^n)}, \quad  t>0.
		\end{equation}
	is valid if one of the following holds
	\begin{enumerate}
		\item  $\beta>n$, and $p\in [r, \infty]$;

		\item  $\beta =n$, and $p\in [r,\infty)$; 

		\item  $\beta <n$, $1\leq r<\frac{n}{\beta}$, and $p\in[r, \frac{r n}{n-r\beta})$; 
			
		\item  $\beta <n$, $ r=\frac{n}{\beta}$, and $p\in [r,\infty)$; or
		\item  $\beta <n$, $ r>\frac{n}{\beta}$, and $p\in [r,\infty]$.	\end{enumerate}

	\end{prop}
	
	\begin{proof} We use  Young's inequality for convolutions to obtain
	\begin{align}\label{S(t)-estimate}
	\|S(t)u_0\|_{L^p}&=\|S_{\alpha,\beta}(t, \cdot)\ast u_0\|_{L^p}\leq \|S_{\alpha,\beta}(t, \cdot)\|_{L^s}\|u_0\|_{L^r}
	\end{align}
	where $r, p, s$ satisfies $1+\frac{1}{p}=\frac{1}{s}+\frac{1}{r}$.
	By Lemma \ref{lem-Ls-estimate}, if $\beta> n$,  we have $S_{\alpha,\beta}(t, \cdot) \in L^s (\mathbb R^n)$ for $s\in [1,\infty]$ with the inequality \eqref{S estimate} valid.  And by $\frac{1}{p}=\frac{1}{s}+\frac{1}{r}-1$, we have $S(t)u_0 \in L^p(\mathbb R^n)$ for $p\in[r,\infty]$. Inserting the estimate of \eqref{S estimate} into \eqref{S(t)-estimate}, we get \eqref{basic1}.
		
	If $\beta=n$, we have $S_{\alpha,\beta}(t, \cdot) \in L^s (\mathbb R^n)$ for $s\in [1,\infty)$, thus $S(t)u_0 \in L^p(\mathbb R^n)$ for $p\in[r,\infty)$.

	If $\beta<n$, we have $S_{\alpha,\beta}(t, \cdot) \in L^s(\mathbb R^n)$ for $ s\in [1, \frac{n}{n-\beta})$. We have the following subcases:	
	\begin{enumerate}[(i)]
		\item If $1\leq r<\frac{n}{\beta}$,  a straightforward calculation yields $p\in[r, \frac{r n}{n-r\beta})$.
				\item if $r=\frac{n}{\beta}$, and $s\in [1, \frac{n}{n-\beta}) $, we calculate that $p\in [r,\infty)$.
		\item if $r>\frac{n}{\beta}$,  and $s\in [1, \frac{n}{n-\beta}) $, this yields $p\in [r,\infty].$
	\end{enumerate} 
	Therefore, our claims are verified.
	\end{proof}
	\begin{rem}In the case $\beta = n$ the estimate is
	\begin{equation*}
			\|S(t)u_0\|_{L^{p }(\mathbb R^n)}\lesssim t^{-\alpha {(\frac{1}{r}-\frac{1}{p} )}} \|u_0\|_{L^{r}(\mathbb R^n)}, \quad  t>0.
		\end{equation*}    
	\end{rem}

	\begin{rem}
		We can rewrite Proposition \ref{Lp estimate for h} as follows \textendash{} the inequality \eqref{basic1} is valid if either
	
	\begin{enumerate}
		 \item $r>\max\{\frac{n}{\beta}, 1\}$, and $p\in [r, \infty]$;		
		 \item  $r=\max\{\frac{n}{\beta}, 1\}$, and $p\in [r,\infty)$; or
		 \item  $1\leq r< \max\{\frac{n}{\beta}, 1\}$, i.e.,  $1\leq r<\frac{n}{\beta}$, and $p\in[r, \frac{r n}{n-r\beta})$.
			\end{enumerate}
	\end{rem}
	
	Next, we are going to give the space-time estimate for $S(t)u_0$. Before doing that, we introduce some notation.
\begin{defn} The triplet $(q, p, r)$ is called an admissible triplet for the equation \eqref{fse} if 
\[\frac{1}{q}=\frac{\alpha n}{\beta}\left(\frac{1}{r}-\frac{1}{p} \right ), \]
where $r, p$ satisfies one of the following restrictions:

	\begin{enumerate}
		\item $\alpha n >\beta$, $1\leq r < \frac{n}{\beta}$, and
		\[1\leq r \leq p <\min \left(\frac{\alpha n r}{\alpha n-\beta}, \frac{rn}{n-r\beta} \right );\]
		\item $\alpha n >\beta$ $r > \frac{n}{\beta}$, and
		\[\frac{n}{\beta}\leq r \leq p <\frac{\alpha n r}{\alpha n-\beta} ;\]
		\item $\alpha n \leq \beta$, $1\leq r < \frac{n}{\beta}$, and
		\[1\leq r \leq p < \frac{rn}{n-r\beta} ; \text{ or}\]
		\item $\alpha n \leq \beta$, $r\geq \max(\frac{n}{\beta}, 1)$, and 
		\[\max\left (\frac{n}{\beta}, 1\right ) \leq r \leq p <\infty .\]

	\end{enumerate}
The triplet $(q, p, r)$ is called a generalized admissible triplet for the equation \eqref{fse} if 
\[\frac{1}{q}=\frac{\alpha n}{\beta}\left (\frac{1}{r}-\frac{1}{p} \right ), \]
and one of the following holds:
	\begin{enumerate}
		\item If $\alpha n<\beta $ or $\alpha n>\beta$ with $r>\frac{n}{\beta}$ and 
		\[1\leq r \leq p <\infty,\]
		\item If $\alpha n>\beta$, $1\leq r\leq \frac{n}{\beta}$, and
		\[1\leq r \leq p < \frac{rn}{n-r\beta}.\]
	\end{enumerate}
	We use $\Pi$ to denote all the admissible triplet and $\tilde \Pi$ all the generalized admissible triplet.
	
	It is easy to notice that for admissible triplets, we have $r<q\leq \infty$; for generalized admissible triplets, we have $1<q\leq \infty$.

\end{defn}	
Let $I=[0, T).$ We define the time-weighted Banach space $\mathcal C_\sigma(I; L^p(\mathbb R^n))$ as follows,

\[\mathcal C_\sigma(I; L^p(\mathbb R^n))=\{f\in C(I;L^p(\mathbb R^n)): \|f \|_{\mathcal C_\sigma(I; L^p(\mathbb R^n))}=\sup_{t\in I}t^{\frac{1}{\sigma}}\|f\|_{L^p(\mathbb R^n)} <\infty\}
\]

	As a corollary of  Proposition \ref{Lp estimate for h}, we obtain space-time estimates of $S(t)u_0$.
	\begin{cor}\label{time space estimate of S(t)}
	\begin{enumerate}
		\item  Assume that $(q, p, r) \in \Pi$ and $u_0 \in L^r (\mathbb R^n)$.  Then 
	\begin{align*}
		\|S(t)u_0\|_{L^{\infty}(I;L^r(\mathbb R^n))}+\|S(t)u_0\|_{L^{q}(I; L^{p}(\mathbb R^n)) }\lesssim \|u_0\|_{L^r(\mathbb R^n)}
	\end{align*}
	\item Assume that $(q, p, r) \in \tilde \Pi$  and $u_0 \in L^r (\mathbb R^n)$.  Then  
	\begin{align*}
		\|S(t)u_0\|_{L^{\infty}(I;L^r(\mathbb R^n))}+\|S(t)u_0\|_{\mathcal C_{q}(I; L^{p}(\mathbb R^n)) }\lesssim \|u_0\|_{L^r(\mathbb R^n)}
	\end{align*}
	\end{enumerate}	
	\end{cor}
	\begin{proof} Statement (2) follows easily from Proposition \ref{Lp estimate for h}, thus we only need to prove (1).
	
	For the case $p=r, q=\infty$, the above estimate also follows from Proposition \ref{Lp estimate for h}. Now we consider the case $p\neq r$, Define the operator $T(t)u_0$ by
	\[T(t)u_0:=\|S(t)u_0\|_{L^ p(\mathbb R^n)}\] 		
	
	Since $(q, p, r)$ is an admissible triplet, we deduce that 
	\[T(t)u_0:=\|S(t)u_0\|_{L^ p(\mathbb R^n)}\leq C t^{-\frac{1}{q}}\|u_0\|_{L^r(\mathbb R^n)}. \]
	It is also easy to check that
\begin{align*}
		\left|\left\{t: |T(t)u_0| >\tau \right \}\right|
		&\leq \left|\left\{t: C t^{-\frac{1}{q}}\|u_0\|_{L^r(\mathbb R^n)} >\tau \right\}\right|\\ 
		&\leq \left|\left\{t: t<\left(\frac{C\|u_0(t)\|_{L^r(\mathbb R^n)}}{\tau} \right)^q\right \}\right|\\
		& \leq \left(\frac{C\|u_0(t)\|_{L^r(\mathbb R^n)}}{\tau} \right)^q,
	\end{align*} 
	which implies that $T(t)$ is a weak $(r, q)$ operator.
	
	On the other hand, $T(t)$ is sub-additive and satisfies the inequality for $r\leq p \leq \infty$, 
	\[T(t)u_0 \leq C \|u_0(t)\|_{L^p(\mathbb R^n)}.\]
	This implies that $T(t)$ is a $(p, \infty)$ operator. Since for any admissible triplet $(q, p, r)$ we can always find another admissible triplet $(q_1, p, r_1)$ such that:
	\[q_1<q<\infty, \ r_1<r<p, \]
	and
	\[\frac{1}{q}=\frac{\gamma}{q_1}+\frac{1-\gamma}{\infty},\ \frac{1}{r}=\frac{\gamma}{r_1}+\frac{1-\gamma}{p}, \]
	then the Marcinkiewicz interpolation theorem implies that $T(t)$ is a strong $(r, q)$ operator.	
	\end{proof}
In the following proposition, we give the space-time estimate for the non-homogeneous part of \eqref{mild solution}.
\begin{prop}\label{time space estimate of G(t)}
For $\theta>0$ and $T>0$, let $r_0=\frac{n\theta}{\beta}$ and $I=[0,T).$ Assume that $r\geq r_0>1$,  $(q,p,r) \in \tilde \Pi$  and $p>\theta+1$. 
 If $f\in L^{\frac{q}{\theta+1}}(I; L^{\frac{p}{\theta+1}}(\mathbb R^n))$, then 
	$Gf \in L^{q}(I; L^{p}(\mathbb R^n)) \cap C_b(I; L^r (\mathbb R^n))$ and we have 
	\begin{align}
		\|Gf\|_{L^{\infty}(I;L^r(\mathbb R^n))}+\|Gf\|_{L^{q}(I; L^{p}(\mathbb R^n)) }\lesssim T^{\alpha-\frac{\alpha n\theta}{\beta r}}\|f\|_{L^{\frac{q}{\theta+1}}(I; L^{\frac{p}{\theta+1}}(\mathbb R^n))},
	\end{align}
	for $p<r(\theta+1)$ and 
   \begin{equation}
		\|Gf\|_{L^{\infty}(I;L^r(\mathbb R^n))}+\|Gf\|_{L^{q}(I; L^{p}(\mathbb R^n)) }\lesssim T^{\alpha-\frac{\alpha n\theta}{\beta r}}\||f|^{\frac{1}{\theta+1}}\|^{\gamma (\theta+1)}_{L^\infty(I; L^r(\mathbb R^n))}\||f|^{\frac{1}{\theta+1}}\|^{(1-\gamma) (\theta+1)}_{L^{q}(I; L^{p}(\mathbb R^n))} 
	\end{equation}
	for $p\geq r(\theta +1)$, where $\gamma= \frac{p-r(\theta+1)}{(p-r)(\theta+1)}$.
	
\end{prop}

\begin{proof} First, we consider the $L^{\infty}(I;L^r(\mathbb R^n))$ norm. For the case $p<r(\theta+1)$, i.e., $\frac{p}{\theta+1}<r$, using(2) of Lemma 
\ref{lem-Ls-estimate}  and Young's inequality, one has 
\begin{align*}
	\|Gf\|_{L^{\infty}(I;L^r(\mathbb R^n))}&\lesssim \sup_{t\in I}\int_0^t \left \|\int_{\mathbb R^n} P_{\alpha,\beta}(t-\tau, x-y) f(\tau,y)dy\right\|_{L^r(\mathbb R^n)}ds \nonumber\\
	&\lesssim \sup_{t\in I}\int_0^t \left \| P_{\alpha,\beta}(t-\tau, \cdot) \right\|_{L^{s}(\mathbb R^n)} \left\|f(\tau,\cdot)\right\|_{L^{\frac{p}{\theta+1}}(\mathbb R^n)}d \tau\nonumber\\
	& \lesssim \sup_{t\in I}\int_0^t (t-\tau)^{\alpha-1-\frac{\alpha n}{\beta}(1-\frac{1}{s}) } \|f(\tau,\cdot)\|_{L^{\frac{p}{\theta+1}}(\mathbb R^n)} d\tau \nonumber\\
\end{align*}
	Here $1+\frac{1}{r}=\frac{1}{s}+\frac{\theta+1}{p}$, replace $1-\frac{1}{s}$ with $\frac{\theta+1}{p}-\frac{1}{r}$ in the above inequality and using  H\"older's inequality with respect to $\tau$, we obtain
	\begin{align*}
	\|Gf\|_{L^{\infty}(I;L^r(\mathbb R^n))}&\lesssim \sup_{t\in I}\int_0^t (t-\tau)^{(\alpha-1-\frac{\alpha n}{\beta}(\frac{\theta+1}{p}-\frac{1}{r}))\chi}d\tau \|f\|_{L^{\frac{q}{\theta+1}}(I; L^{\frac{p}{\theta+1}}(\mathbb R^n))},\nonumber\\
	& \lesssim  T^{\alpha-\frac{\alpha n}{\beta}(\frac{\theta+1}{p}-\frac{1}{r})-\frac{\theta+1}{q}} \|f\|_{L^{\frac{q}{\theta+1}}(I; L^{\frac{p}{\theta+1}}(\mathbb R^n))}\\
	& \lesssim T^{\alpha-\frac{\alpha n \theta}{\beta r}}\|f\|_{L^{\frac{q}{\theta+1}}(I; L^{\frac{p}{\theta+1}}(\mathbb R^n))}.
	\end{align*}
	where $\frac{1}{\chi}+\frac{\theta+1}{q}=1$.

For the case $p\geq r(\theta+1)$, notice that $r\leq r(\theta+1)\leq p$, by virtue of  Lebesgue interpolation theorem and  H\"older's inequality, we have:
\begin{align*}
	\|Gf\|_{L^{\infty}(I;L^r(\mathbb R^n))}&\leq \sup_{t\in I}\int_0^t(t-\tau)^{\alpha-1}\||f(\tau,x)|^{\frac{1}{\theta+1}}\|_{L^{r(\theta +1)}}^{\theta+1}d\tau \nonumber \\
	& \lesssim \sup_{t\in I}\int_{0}^t (t-\tau)^{\alpha-1}\||f(\tau,x)|^{\frac{1}{\theta+1}}\|_{L^r}^{(\theta+1)\gamma}\||f(\tau,x)|^{\frac{1}{\theta+1}}\|_{L^p}^{(\theta+1)(1-\gamma)}d\tau \nonumber \nonumber \\
	& \lesssim \||f|^{\frac{1}{\theta+1}}\|_{C(I; L^r)}^{(\theta+1)\gamma} \int_0^t (t-\tau)^{\alpha-1}\||f(\tau,x)|^{\frac{1}{\theta+1}}\|_{L^p}^{(\theta+1)(1-\gamma)}d\tau,
\end{align*}
where $\gamma$ satisfies $\frac{1}{r(\theta+1)}=\frac{\gamma}{r}+\frac{1-\gamma}{p}$, hence $\gamma=\frac{p-r(\theta+1)}{(p-r)(\theta+1)}$. Using H\"older's inequality for the last integral, we have 
\begin{align*}
\|Gf\|_{L^{\infty}(I;L^r(\mathbb R^n))}& \lesssim T^{\alpha-\frac{(1+\theta)(1-\gamma)}{q}}\left \||f|^{\frac{1}{\theta+1}}\right \|_{C(I; L^r)}^{(\theta+1)\gamma} \left \||f|^{\frac{1}{\theta+1}}\right \|_{L^q(I;L^p)}^{(\theta+1)(1-\gamma)}\nonumber\\
&\lesssim T^{\alpha - \frac{\alpha n \theta}{\beta r}}\left \||f|^{\frac{1}{\theta+1}}\right\|_{C(I; L^r)}^{(\theta+1)\gamma} \left \||f|^{\frac{1}{\theta+1}}\right \|_{L^q(I;L^p)}^{(\theta+1)(1-\gamma)}.
\end{align*}

Next, we give the estimate for the $L^q(I; L^p)$ norm.
For the case $p<r(\theta+1)$, by Young's inequality that:
\begin{align*}
	\|G f\|_{L^q(I; L^p)}&\lesssim \left\|\int_0^T (T-\tau)^{\alpha-1-\frac{n\alpha}{\beta}(\frac{\theta+1}{p}-\frac{1}{p})}\|f\|_{\frac{p}{\theta+1}}d\tau\right\|_{L^q}\nonumber\\
	&\lesssim \left(\int_0^T \tau ^{(\alpha-1-\frac{n\alpha}{\beta}\frac{\theta}{p}))\gamma}d\tau\right)^{\frac{1}{\gamma}}\|f\|_{L^{\frac{q}{\theta+1}}(I; L^{\frac{p}{\theta+1}})}\nonumber\\
	&\lesssim T^{\alpha-\frac{n\alpha}{\beta}\frac{\theta}{p}-\frac{\theta}{q}}\|f\|_{L^{\frac{q}{\theta+1}}(I; L^{\frac{p}{\theta+1}})}\\
	& \lesssim T^{\alpha-\frac{n\alpha \theta}{\beta r}}\|f\|_{L^{\frac{q}{\theta+1}}(I; L^{\frac{p}{\theta+1}})}
\end{align*}
where $1+\frac{1}{q}=\frac{1+\theta}{q}+\frac{1}{\gamma}$.

For the case $p\geq r(\theta+1)$, arguing similarly as in the proof above gives:
\begin{align*}
	\|G f\|_{L^q(I; L^p)}&\lesssim \left\|\int_0^T (T-\tau)^{\alpha-1-\frac{n\alpha}{\beta}(\frac{1}{r}-\frac{1}{p})}\||f(\tau, x)|^{\frac{1}{\theta+1}}\|^{\theta+1}_{L^{r(\theta+1)}}d \tau\right\|_{L^q}\nonumber\\
	&\lesssim \left\|\int_0^T (T-\tau)^{\alpha-1-\frac{n\alpha}{\beta}(\frac{1}{r}-\frac{1}{p})}\||f(\tau, x)|^{\frac{1}{\theta+1}}\|^{(\theta+1)\gamma}_{L^{r}}\||f(\tau, x)|^{\frac{1}{\theta+1}}\|^{(\theta+1)(1-\gamma)}_{L^{p}}d \tau \right\|_{L^q}\nonumber\\
	&\lesssim \left(\int_0^T \tau^{(\alpha-1-\frac{n\alpha}{\beta}(\frac{1}{r}-\frac{1}{p}))\chi}d\tau \right)^{\frac{1}{\chi}}\||f(\tau, x)|^{\frac{1}{\theta+1}}\|^{(\theta+1)\gamma}_{C(I:L^{r})}\||f(\tau, x)|^{\frac{1}{\theta+1}}\|^{(\theta+1)(1-\gamma)}_{L^q(I:L^{p})}\nonumber\\
	&\lesssim  T^{\alpha - \frac{\alpha n \theta}{\beta r}}\left \||f(\tau, x)|^{\frac{1}{\theta+1}}\right \|^{(\theta+1)\gamma}_{C(I:L^{r})}\left \||f(\tau, x)|^{\frac{1}{\theta+1}}\right \|^{(\theta+1)(1-\gamma)}_{L^q(I:L^{p})}
\end{align*}
where $\gamma$ and $\chi$ satisfy that:
\[\frac{1}{r(\theta+1)} =\frac{\gamma}{r}+\frac{1-\theta}{p}, 1+\frac{1}{q}=\frac{(\theta+1)(1-\gamma)}{q}+\frac{1}{\chi}, \]
which is meaningful by the fact that $r<r(b+1)<p$.

\end{proof}

Similarly, we have the following result.
 \begin{prop}
 For $\theta>0$ and $T>0$, let $r_0=\frac{n\theta}{\beta}$ and $I=[0,T).$ Assume that $r\geq r_0>1$, $(q,p,r) \in \tilde \Pi$  and $p>\theta+1$. 
 If $f\in \mathcal C_{\frac{q}{\theta+1}}(I; L^{\frac{p}{\theta+1}}(\mathbb R^n))$, then 
	$Gf \in \mathcal C_{q}(I;  L^{p}(\mathbb R^n)) \cap C_b(I; L^r (\mathbb R^n))$ and we have 
	\begin{align}
		\|Gf\|_{L^{\infty}(I;L^r(\mathbb R^n))}+\|Gf\|_{\mathcal C_{q}(I;  L^{p}(\mathbb R^n)) }\lesssim T^{\alpha-\frac{\alpha n\theta}{\beta r}}\|f\|_{\mathcal C_{\frac{q}{\theta+1}}(I; L^{\frac{p}{\theta+1}}(\mathbb R^n))},
	\end{align}
	for $p<r(\theta+1)$ and 
   \begin{equation}
		\|Gf\|_{L^{\infty}(I;L^r(\mathbb R^n))}+\|Gf\|_{\mathcal C_{q}(I;  L^{p}(\mathbb R^n)) }\lesssim T^{\alpha-\frac{\alpha n\theta}{\beta r}}\||f|^{\frac{1}{\theta+1}}\|^{\gamma (\theta+1)}_{L^\infty(I; L^r(\mathbb R^n))}\||f|^{\frac{1}{\theta+1}}\|^{(1-\gamma) (\theta+1)}_{\mathcal C_{\frac{q}{\theta+1}}(I;  L^{\frac{p}{\theta+1}}(\mathbb R^n))} 
	\end{equation}
	for $p\geq r(\theta +1)$, where $\gamma= \frac{p-r(\theta+1)}{(p-r)(\theta+1)}$.
 \end{prop}   
 \begin{proof} The proof is similar to the proof in Proposition 4.7.
 
In fact, for the case $p \leq r(b+1)$, one has
\begin{align*}
	\|Gf\|_{L^{\infty}(I;L^r(\mathbb R^n))}&\lesssim \sup_{t\in I}\int_0^t (t-\tau)^{\alpha-1-\frac{\alpha n}{\beta}(\frac{\theta+1}{p}-\frac{1}{r})}\|f(\tau, x)\|_{L^{\frac{p}{\theta+1}}} d\tau \\
	&\lesssim \sup_{t\in I}\int_0^t (t-\tau)^{\alpha-1-\frac{\alpha n}{\beta} (\frac{\theta+1}{p}-\frac{1}{r})}\tau^{-\frac{\theta+1}{
	q}}d\tau\|f\|_ {\mathcal C_{\frac{q}{\theta+1}}(I:L^{\frac{p}{\theta+1}})}\\
	&\lesssim T^{\alpha-\frac{\alpha n\theta}{\beta r}}\|f\|_ {\mathcal C_{\frac{q}{\theta+1}}(I:L^{\frac{p}{\theta+1}})}.
\end{align*}
Making use of Young's inequality, we have
\begin{align*}
	\|Gf\|_{\mathcal C_{q}(I;  L^{p}(\mathbb R^n)) }&\lesssim \sup_{t\in I} t^{\frac{1}{q}}\int_0^t(t-\tau)^{\alpha-1-\frac{\alpha n}{\beta}(\frac{\theta+1}{p}-\frac{1}{p})}\|f(\tau, x)\|_{L^{\frac{p}{\theta+1}}} d\tau\\
	&\lesssim \sup_{t\in I} t^{\frac{1}{q}}\int_0^t (t-\tau)^{\alpha-1-\frac{\alpha n}{\beta} (\frac{\theta}{p})}\tau^{-\frac{\theta+1}{
	q}}d\tau\|f\|_ {\mathcal C_{\frac{q}{\theta+1}}(I:L^{\frac{p}{\theta+1}})}\\
	&\lesssim T^{\alpha-\frac{\alpha n\theta}{\beta r}}\|f\|_ {\mathcal C_{\frac{q}{\theta+1}}(I:L^{\frac{p}{\theta+1}})}.
\end{align*}

As for the case $p\geq r(\theta+1)$, let $\frac{1}{r(\theta+1)} =\frac{\gamma}{r}+\frac{1-\theta}{p}$ we have 
\begin{align*}
\|Gf\|_{L^{\infty}(I;L^r(\mathbb R^n))}&\leq \sup_{t\in I}\int_0^t(t-\tau)^{\alpha-1}\||f(\tau,x)|^{\frac{1}{\theta+1}}\|_{L^{r(\theta +1)}}^{\theta+1}d\tau \nonumber \\
	& \lesssim \sup_{t\in I}\int_{0}^t (t-\tau)^{\alpha-1}\||f(\tau,x)|^{\frac{1}{\theta+1}}\|_{L^r}^{(\theta+1)\gamma}\||f(\tau,x)|^{\frac{1}{\theta+1}}\|_{L^p}^{(\theta+1)(1-\gamma)}d\tau \nonumber \nonumber \\
	& \lesssim \||f|^{\frac{1}{\theta+1}}\|_{C(I; L^r)}^{(\theta+1)\gamma} \sup_{t\in I}\int_0^t (t-\tau)^{\alpha-1}\||f(\tau,x)|^{\frac{1}{\theta+1}}\|_{L^p}^{(\theta+1)(1-\gamma)}d\tau\\
	&\lesssim \||f|^{\frac{1}{\theta+1}}\|_{C(I; L^r)}^{(\theta+1)\gamma} \||f(\tau,x)|^{\frac{1}{\theta+1}}\|_{\mathcal C_q (I:{L^p})}^{(\theta+1)(1-\gamma)}\sup_{t\in I}\int_0^t (t-\tau)^{\alpha-1}\tau^{-\frac{\theta+1}{q}(1-\gamma)} d\tau\\
	&\lesssim T^{\alpha-\frac{\alpha n\theta}{\beta r}}\||f|^{\frac{1}{\theta+1}}\|_{C(I; L^r)}^{(\theta+1)\gamma} \||f(\tau,x)|^{\frac{1}{\theta+1}}\|_{\mathcal C_q (I:{L^p})}^{(\theta+1)(1-\gamma)},
\end{align*}
and 
\begin{align*}
	\|G f\|_{\mathcal C_q(I; L^p)}&\lesssim \sup_{t\in I} t^{\frac{1}{q} }\int_0^t (t-\tau)^{\alpha-1-\frac{n\alpha}{\beta}(\frac{1}{r}-\frac{1}{p})}\||f(\tau, x)|^{\frac{1}{\theta+1}}\|^{\theta+1}_{L^{r(\theta+1)}}d \tau\nonumber\\
	&\lesssim \sup_{t\in I} t^{\frac{1}{q} } \int_0^t (t-\tau)^{\alpha-1-\frac{n\alpha}{\beta}(\frac{1}{r}-\frac{1}{p})}\||f(\tau, x)|^{\frac{1}{\theta+1}}\|^{(\theta+1)\gamma}_{L^{r}}\||f(\tau, x)|^{\frac{1}{\theta+1}}\|^{(\theta+1)(1-\gamma)}_{L^{p}}d \tau\nonumber\\
	&\lesssim \sup_{t\in I} t^{\frac{1}{q} } \int_0^t (t-\tau)^{\alpha-1-\frac{n\alpha}{\beta}(\frac{1}{r}-\frac{1}{p})}\tau^{-\frac{(\theta+1)(1-\gamma)}{q}}d \tau \||f(\tau, x)|^{\frac{1}{\theta+1}}\|^{(\theta+1)\gamma}_{L^{r}}\||f(\tau, x)|^{\frac{1}{\theta+1}}\|^{(\theta+1)(1-\gamma)}_{{\mathcal C_q(I:L^{p})}}\nonumber\\
	&\lesssim  T^{\alpha - \frac{\alpha n \theta}{\beta r}}\||f(\tau, x)|^{\frac{1}{\theta+1}}\|^{(\theta+1)\gamma}_{C(I:L^{r})}\||f(\tau, x)|^{\frac{1}{\theta+1}}\|^{(\theta+1)(1-\gamma)}_{\mathcal C_q(I:L^{p})}.
\end{align*}
Thus, we have proved the proposition.

 \end{proof}

\section{Local well-posedness }
In this section, we study the Cauchy problem of \eqref{fse} for the initial data $u_0\in L^r (\mathbb R^n)$, with $1\leq r_0= \frac{n\theta}{\beta}\leq r$. Thus we consider the integral equation \eqref{mild solution}.

First consider the solution in the space 
$X(I)=C_b(I; L^r(\mathbb R^n))\cap L^q(I; L^p(\mathbb R^n)),$ where $I=[0, T)$ for $T>0$. 
Now we have following theorem.
\begin{thm} Let $1<r_0=\frac{n\theta}{\beta}\leq r$ and $u_0\in L^r$. Assume that $(q,p,r)$ is an arbitrary admissible triplet.
\begin{enumerate}
	\item There exist $T>0$ and a unique mild solution $u\in X(I)$, where $T= T(\|u\|_{L^r})$ depends on the norm $\|u_0\|_{L^r}$ for $r>r_0$, and $T=T(u_0)$ depends on $u_0$ when $r=r_0$.
	\item If $r=r_0,$ then $T=\infty$ provided $\|u_0\|_{L^r}$ is sufficiently small.  In other words, there exists a global small solution $u\in C_b([0,\infty); L^r(\mathbb R^n))\cap L^q([0,\infty); L^p(\mathbb R^n)) $
\end{enumerate}
	
\end{thm}
\begin{proof} Now we define the space $X(I)=C(I; L^r(\mathbb R^n))\cap L^q(I; L^p(\mathbb R^n))$ with the norm 
\[\|u\|_{X(I)}:=\sup_{t\in I} \|u\|_{L^\infty (I; L^r)}+\sup_{(q,p, r) \in  \Pi}\|u\|_{L^q(I; L^p(\mathbb R^n))} .\]
Denote $\mathcal T u= S(t)u_0+ Gf(u)$ is a nonlinear operator. Now we consider the fixed point for $\mathcal T$ on the following Banach space 
\[\mathcal X(I)=\{u(t)\in X(I); \|u(t)\|_{X(I)}\leq C\|u_0\|_{L^r(\mathbb R^n)}\},\]
	with $d(u, v)=\|u-v\|_{X(I)}$.
	
Combining the  Corollary \ref{time space estimate of S(t)} and Proposition \ref{time space estimate of G(t)}, we have the following result
\begin{align*}
	\|\mathcal T u\|_{X(I)}\lesssim \|u_0\|_{X(I)}+ T^{\alpha-\frac{\alpha n\theta}{\beta r}}\|u_0\|_{X(I)}^{\theta+1},
\end{align*} thus $\mathcal T$ is a linear operator from $\mathcal X(I)$ to $\mathcal X(I)$.	
It is straightforward to get 
\begin{align*}
	d(\mathcal T(u), \mathcal T(v))&=\sup_{t\in I} \|\mathcal T(u)- \mathcal T(v)\|_{L^\infty (I; L^r)}+\sup_{(q,p, r) \in  \Pi}\|\mathcal T(u)-\mathcal T(v)\|_{L^q(I; L^p(\mathbb R^n))}\\
	&\leq \sup_{t\in I} \| Gf(u)-  Gf(v)\|_{L^\infty (I; L^r)}+\sup_{(q,p, r) \in  \Pi}\|Gf(u)-  Gf(v))\|_{L^q(I; L^p(\mathbb R^n))}\\
	&\leq \sup_{t\in I} \big{\|} \int_0^t \int_{\mathbb R^n} P(t-\tau, x-y)(|u(\tau, y)|^\theta u(\tau, y)- |v(\tau, y)|^\theta v(\tau, y)) dy d \tau \big{\|}_{L^\infty (I; L^r)}\\
	&+\sup_{(q,p, r) \in \Pi}\big{\|} \int_0^t \int_{\mathbb R^n} P(t-\tau, x-y)(|u(\tau, y)|^\theta u(\tau, y)- |v(\tau, y)|^\theta v(\tau, y)) dy d \tau \big{\|}_{L^q(I; L^p(\mathbb R^n))},\\
	&:= I_1+I_2
\end{align*}
 since $ ||u|^\theta u- |v|^\theta v| \leq (|u|^\theta +|v|^\theta)(u-v)$ and by the H\"older's inequality, we have that 
\[\left\| |u|^\theta u- |v|^\theta v\right\|_{L^{r}(\mathbb R^n)}\leq (\|u\|_{L^\infty(\mathbb R^n)}^\theta+\|v\|_{L^\infty(\mathbb R^n)}^\theta )\|u-v\|_{L^r(\mathbb R^n)}.\]
Then we have 
\begin{align*}	
I_1&=\left\| \int_0^t \int_{\mathbb R^n} P(t-\tau, x-y)\left(|u(\tau, y)|^\theta u(\tau, y)- |v(\tau, y)|^\theta v(\tau, y)\right) dy d \tau \right\| _{L^\infty (I; L^r)}\\
& \leq \sup_{t\in I}\int_0^t \left \|\int_{\mathbb R^n} P(t-\tau, x-y)\left (|u(\tau, y)|^\theta u(\tau, y)- |v(\tau, y)|^\theta v(\tau, y)\right ) dy\right \| _{L^r} d\tau\\
&\leq \sup_{t\in I} \int_0^t (t-\tau)^{\alpha-1}\left(\|u(\tau, \cdot)\|_{L^\infty(\mathbb R^n)}^\theta+\|v(\tau, \cdot)\|_{L^\infty (\mathbb R^n)}^\theta \right)\|u(\tau, \cdot)-v(\tau, \cdot)\|_{L^r(\mathbb R^n)} d\tau\\
&\leq  \sup_{t\in I} \|t^{\alpha-1}\|_{L^{\frac{1}{1-\frac{\alpha n \theta}{\beta r }}}}\left(\|u\|_{L^{\frac{\beta r}{\alpha n \theta}}(I:L^\infty(\mathbb R^n) }^\theta+\|v\|_{L^{\frac{\alpha n}{\theta \beta r}}(I:L^\infty (\mathbb R^n)}^\theta \right )\|u-v\|_{L^{\infty}(I: L^r(\mathbb R^n))}\\
& \leq T^{\alpha-\frac{\alpha n \theta}{\beta r}}\left(\|u\|_{X(I)}^{\theta}+\|v\|_{X(I)}^{\theta}\right)d(u,v).
\end{align*}
and
\begin{align*}
	I_2&=\sup_{(q,p, r) \in  \Pi}\left \| \int_0^t \int_{\mathbb R^n} P(t-\tau, x-y)\left(|u(\tau, y)|^\theta u(\tau, y)- |v(\tau, y)|^\theta v(\tau, y)\right) dy d \tau \right \|_{L^q(I; L^p)}\\
& \leq \sup_{(q,p, r) \in  \Pi} \left \|\int_0^t \left \|\int_{\mathbb R^n} P(t-\tau, x-y)(|u(\tau, y)|^\theta u(\tau, y)- |v(\tau, y)|^\theta v(\tau, y)) dy\right \|_{L^p} d\tau \right \|_{L^q}\\
&\leq \sup_{(q,p, r) \in  \Pi}\left \| \int_0^t (t-\tau)^{\alpha-1}(\|u(\tau, \cdot)\|_{L^\infty(\mathbb R^n)}^\theta+\|v(\tau, \cdot)\|_{L^\infty (\mathbb R^n)}^\theta )\left\| u(\tau, \cdot)-v(\tau, \cdot)\right  \|_{L^p(\mathbb R^n)}d\tau\right \|_{L^q} \\
&\leq \sup_{(q,p, r) \in  \Pi} \|t^{\alpha-1}\|_{L^{\frac{1}{1-\frac{\alpha n \theta}{\beta r }}}(I)}(\|u\|_{L(I)^{\frac{\alpha n}{\theta \beta r}}(I:L^\infty(\mathbb R^n))}^\theta+\|v\|_{L^{\frac{\alpha n}{\theta \beta r}}(I:L^\infty (\mathbb R^n)}^\theta )\|u-v\|_{L^{q}(I: L^p(\mathbb R^n))}\\
& \leq T^{\alpha-\frac{\alpha n \theta}{\beta r}}\left(\|u\|_{X(I)}^{\theta}+\|v\|_{X(I)}^{\theta}\right )d(u,v).
\end{align*}
	hence we have that 
	\[d(\mathcal T(u), \mathcal T(v))\leq  C T^{\alpha-\frac{\alpha n \theta}{\beta r}}\|u_0\|_{L^r(\mathbb R^n)}^\theta d(u,v).\]
Obviously, we can choose $T$ small enough to guarantee that the operator $\mathcal T$ is contract on $\mathcal X(I)$, i.e.,
\begin{align}\label{time ineq}
T \leq  C \|u_0\|_{L^r(\mathbb R^n)}^{\frac{\alpha n}{\beta r}-\frac{\alpha}{\theta}}.
\end{align}
Thus by the Banach contraction mapping theorem, there exist one unique solution $u(t)\in X(I)$.

(2) It is obvious from inequality \eqref{time ineq}.
\end{proof}
Finally we consider the solution in the space 
$Y(I)=C_b(I; L^r(\mathbb R^n))\cap \mathcal C_q(I; L^p(\mathbb R^n)),$ where $I=[0, T)$ for $T>0$. Similarly, we have the following result:

\begin{thm} Let $1<r_0=\frac{n\theta}{\beta}\leq r$ and $u_0\in L^r(\mathbb R^n)$. Assume  that$(q,p,r)\in \tilde \Pi$.
\begin{enumerate}
	\item There exist $T>0$ and a unique mild solution $u\in Y(I)$, where $T= T(\|u\|_{L^r (\mathbb R^n)})$ depends on the norm $\|u_0\|_{L^r(\mathbb R^n)}$ for $r>r_0$, and $T=T(u_0)$ depends on $u_0$ when $r=r_0$.
	\item If $r=r_0,$ then $T=\infty$ provided $\|u_0\|_{L^r (\mathbb R^n)}$ is sufficiently small.  In other words, there exists a global small solution $u\in C_b([0,\infty); L^r(\mathbb R^n))\cap \mathcal C_q([0,\infty); L^p(\mathbb R^n)) $.
	\end{enumerate}
	
\end{thm}

\begin{rem} We cannot get the blow up criterion for this time fractional equation like what we can do for the integer case, since the equation lacks of the time-translation invariance property. In other words, we cannot use the iterate method to extend the existence time for our solution due to its loss of time-translation invariance. 	
\end{rem}

\appendix
\renewcommand*{\thesection}{\Alph{section}} 
\section{Properties of $H$-functions}

In this section, we give the asymptotic expansion as $H_{p,q}^{m,n}(z)$ for $z$ goes to infinity  or 0 and $\arg z$ lies in a certain sector.
Since in this paper, we mainly care about the results about $H^{2,1}_{3,2}(z)$, thus tailor the classical theorems to the special cases.

\begin{lem} \label{A1}
Assume that either $\Lambda\leq 0$ or $\Lambda>0, a^*>0$ with additional condition $|\arg z|< a^*\pi /2$. Then the asymptotic expansion of $H_{2,3}^{2,1}(z)$  near infinity is given by 
\[H_{2,3}^{2,1}(z)\sim \sum_{k=0}^{\infty}h_{k}z^{({a_1-1-k})/
{\alpha_1}}, \quad z \to \infty , \]
\end{lem}
 and 
\[h_k=\frac{1}{\alpha_1}\frac{\Gamma(b_1-(a_1-1-k)\frac{\beta_1}{\alpha_1})\Gamma(b_2-(a_1-1-k)\frac{\beta_2}{\alpha_1})}{\Gamma(a_2-(a_1-1-k)\frac{\alpha_2}{\alpha_1})\Gamma(1-b_3+(a_1-1-k)\frac{\beta_3}{\alpha_1})}. \]

The next two lemmas give the asymptotic expansion of $H_{2,1}^{2,3}(z)$ near zero under the assumptions whether the poles of $\Gamma(b_1+\beta_1 s)$ and $\Gamma(b_2+\beta_2 s)$ coincide.
\begin{lem} \label{A2} Assume that the poles of $\Gamma(b_1+\beta_1 s)$ and $\Gamma(b_2+\beta_2 s)$  do not coincide, and  let $\Lambda>0$. Then the asymptotic expansion of $H_{2,1}^{2,3}(z)$ near zero is given by
\[H^{2,1}_{2,3}(z)\sim \sum_{j=1}^2\sum_{l=0}^{\infty} h_{jl}z^{(b_j+l)/\beta_j},\quad z\to 0, \]
where $h^*_{jl}$ is given by 
\[h^*_{1l}=\frac{(-1)^l}{l!\beta_1}\frac{\Gamma(b_2-[b_1+l]\frac{\beta_2}{\beta_1} )\Gamma(1-a_1+[b_1+l]\frac{\alpha _3}{\beta_1} )}{\Gamma(a_3-[b_1+l]\frac{\alpha_3}{\beta_1})\Gamma(1-b_3+[b_1+l]\frac{\beta_3}{\beta_1})}, \]
and \[h^*_{2l}=\frac{(-1)^l}{l!\beta_2}\frac{\Gamma(b_1-[b_2+l]\frac{\beta_1}{\beta_2} )\Gamma(1-a_1+[b_2+l]\frac{\alpha _3}{\beta_2} )}{\Gamma(a_3-[b_2+l]\frac{\alpha_3}{\beta_2})\Gamma(1-b_3+[b_2+l]\frac{\beta_3}{\beta_2})}. \]
In particular, the principal terms of this asymptotic have the form 
\[H^{2,1}_{2,3}(z)\sim \sum_{j=1}^2 h_{j^*}z^{b_j/\beta_j}+o(z^{b_j/\beta_j}),(z\to 0), \]
	where \[h^*_{1}=\frac{1}{\beta_1}\frac{\Gamma(b_2-b_1\frac{\beta_2}{\beta_1} )\Gamma(1-a_1+b_1\frac{\alpha _3}{\beta_1} )}{\Gamma(a_3-b_1\frac{\alpha_3}{\beta_1})\Gamma(1-b_3+b_1\frac{\beta_3}{\beta_1})}, \]
and \[h^*_{2}=\frac{1}{\beta_2}\frac{\Gamma(b_1-b_2\frac{\beta_1}{\beta_2} )\Gamma(1-a_1+b_2\frac{\alpha _3}{\beta_2} )}{\Gamma(a_3-b_2\frac{\alpha_3}{\beta_2})\Gamma(1-b_3+b_2\frac{\beta_3}{\beta_2})}. \]
\end{lem}
As for the case that the poles of $\Gamma(b_1+\beta_1 s)$ and $\Gamma(b_2+\beta_2 s)$ coincide, we give the result in the following theorem.
\begin{lem} \label{A3}Assume that the poles of $\Gamma(b_1+\beta_1 s)$ and $\Gamma(b_2+\beta_2 s)$  do coincide, and  let either $\Lambda>0$ or $\Lambda<0, a^*>0$ with additional condition $|\arg z|<\frac{a^*\pi}{2}$. Then the principal terms of the asymptotic expansion of $H_{2,1}^{2,3}(z)$ near zero is given by
\[H^{2,1}_{2,3}(z)\sim H^{*}z^{\frac{b_1}{\beta_1}}[\ln z]^{N^*-1}+o(z^{\frac{b_1}{\beta_1}}[\ln z]^{N^*-1}), \quad z\to 0,\]
where $N^*$ is the order of the poles.	
\end{lem}

\begin{lem} \label{A4}Let $n=0$, $\Lambda>0, a^* >0$, and $\varepsilon$ be a constant such that $0<\varepsilon<\frac{\Lambda \pi}{2}  $, then for the $H$-function $H^{q,0}_{p,q}(z)$, the following asymptotic estimate holds at infinity:
\begin{align}
	H^{q,0}_{p,q}(z)=O\left (|z|^{[\Re(\mu)+1/2]/\Lambda}\exp \left
	\{\Lambda(\frac{|z|}{\delta} )^{1/\Lambda}\max \left(\cos\frac{a_1^*\pi +\arg z}{\Lambda},\cos\frac{a_1^*\pi -\arg z}{\Lambda}  \right)  \right\}\right),\quad z\to \infty 
\end{align}
uniformly on $|\arg z|\leq \frac{\Lambda \pi}{2}-\varepsilon$.
\end{lem}

\section*{Acknowledgement}
This project was supported by the NSFC-RFBR Programme of China (No. 11611530677).

\end{document}